\documentclass[a4paper,11pt]{article}          
\usepackage{amsfonts,amsmath,latexsym,amssymb} 
\usepackage{amsthm}      
\usepackage[dvipsnames]{xcolor}          

\usepackage{float}
\usepackage{tikz}
\usepackage{circuitikz}
\usepackage{hyperref}
\usepackage{mathtools}
\usepackage{enumitem}
\usepackage[]{mdframed}
\usepackage[numbers,square]{natbib}
\allowdisplaybreaks

  \evensidemargin
\oddsidemargin
\newtheorem{theorem}{Theorem}   
\newtheorem{remark}[theorem]{Remark}
\newtheorem{proposition}[theorem]{Proposition} 
\newtheorem{lemma}[theorem]{Lemma}               
\newtheorem{corollary}[theorem]{Corollary}

\theoremstyle{definition}

\newcommand{\E}{\mathbb{E}}

\def\CC{\mathbb{C}}

\def\EE{\mathbb{E}}

\def\NN{\mathbb{N}}

\def\PP{\mathbb{P}}

\def\RR{\mathbb{R}}
\def\RRd1{\mathbb{R}^{d+1}}

\def\vol{\textup{vol}}
\def\conv{\textup{conv}}

\begin{document}

\title{\bfseries On the moments of the volume for random convex chains}

\author{Anna Gusakova\footnotemark[1],\; Anna Muranova\footnotemark[2]}

\date{}
\renewcommand{\thefootnote}{\fnsymbol{footnote}}
\footnotetext[1]{University of Münster, Institut für Mathematische Stochastik, Orl\'eans-Ring 10, 48149 Münster, Germany Email: gusakova@uni-muenster.de}
\footnotetext[2]{Uniwersytet Warmińsko-Mazurski, 
Wydział Matematyki i Informatyki,
ul. Słoneczna 54, 
10-710 Olsztyn, Poland. Email: anna.muranova@matman.uwm.edu.pl}

\maketitle

\begin{abstract}
Let $T$ be the triangle in the plane with vertices $(0, 0)$, $(0,1)$ and $(0, 1)$. The convex
hull $T_n$ of points $(0, 1)$, $(1, 0)$ and $n$ independent random points uniformly distributed in $T$ is
the random convex chain. In this paper we study the moments of the volume of random polytope $T_n$ and derive exact formulas for $k$-th moments for any integer $k\ge 0$. As an intermediate result, we find an explicit representation for the probability generating function of the number of vertices of $T_n$, from which an alternative formula for the probability that $T_n$ has $k$ vertices follows.\\

\end{abstract}

\section{Introduction}

Let $K\subset \RR^d$ be a convex body and let $X_1,\ldots,X_n$ be independent random points uniformly distributed in $K$. Denote by $K_n$ the convex hull of $X_1,\ldots,X_n$, which is a random convex polytope contained in $K$. This model of random polytope has a very long and extensive history of studying going back to the work of Renyi and Sulanke \cite{RS63} and Sylvester's four point problem \cite{Syl64, Pfi89}. For an overview on the topic we refer reader to \cite{Schneider2008} and references therein. 

Among the most studied functionals of $K_n$ is the volume $\vol(K_n)$. It appeared to be very hard problem to find an explicit expression for the moments $\EE[\vol(K_n)^k]$, $k\in\mathbb{N}$. Even for the expected volume $\EE[\vol(K_n)]$ there are just a few results available in the literature, namely in case when $K$ is a ball \cite{Aff88}, $K$ is a planar triangle or parallelogram \cite{Buch84}, $K$ is a three-dimensional simplex \cite{BR01}. For moments of higher order the exact formulas exist only for $n=d+1$, in which case $K_{d+1}$ is a random simplex. Again in this case only some special types of convex bodies $K$ can be treated (see i.e. \cite{Beck24} and references therein). Another functional of interest is the number of vertices $f_0(K_n)$ of the random polytope $K_n$. The first moments of $\vol(K_n)$ and $f_0(K_n)$ are related via classical Efron's identity \cite{Efr65}
\begin{equation}\label{eq:Efron}
{\EE[\vol(K_n)]\over \vol(K)}=1-{\EE[f_0(K_{n+1})]\over n+1}.
\end{equation}
A generalized form of this identity for higher moments was obtained by Buchta \cite{Buch05, Buch21}
\[
{\EE[\vol(K)^k]\over \vol(K)^k}=\EE\Big[\prod_{i=1}^k\left(1-{ f_0(K_{n+k})\over n+i}\right)\Big].
\]

The reason for the lack of exact formulas is the fact that the distributions of random variables $\vol(K_n)$ and $f_0(K_n)$ depend strongly on the boundary structure of the contained body $K$. In case of an arbitrary $K$ it is basically impossible to derive an exact expression even for $\EE[\vol(K_n)]$ and the only treatable cases are the convex bodies $K$, whose boundary has sufficiently simple structure, like ball or simplex. While the exact expressions are often too complicated to obtain, an alternative question of studying the asymptotic probabilistic behavior of $\vol(K_n)$ or, more precisely, of $\vol(K\setminus K_n)$, as $n\to\infty$ appeared to be more treatable, resulting in a big amount of articles (see i.e. \cite{Schneider2008} and references therein). Despite this the exact asymptotic for $\EE[\vol(K\setminus K_n)^k]$ is known only for $k=1$ and $k=2$ and in case of some special families of convex bodies $K$, like smooth convex bodies or polytopes.

\paragraph{The model.} In this article we consider a special model of random polygones, which appeared to be more accessible for exact computations. Let $T\subset \RR^2$ be a triangle with vertices $(0,1)$, $(0,0)$ and $(1,0)$ and let $X_1,\ldots,X_n$, $n\in\NN$ be independent random points, uniformly distributed in $T$. Consider a random polygone $T_n$, which is the convex hull of points $X_1,\ldots, X_n$ together with the corners $(0,1)$ and $(1,0)$ of triangle $T$. The sequence of vertices of $T_n$ forms a random convex chain and in this article we will abuse the notation slightly and use the name random convex chain to denote the random polygone $T_n$ itself. In \cite{Buchta06} Buchta has determined the precise distribution of the number $N_n:=f_0(T_n)-2$ of "true" vertices of $T_n$, namely he showed that for $k=1,\ldots,n$ it holds
\begin{equation}\label{eq:Probabilities}
\begin{split}
p_k^{(n)}:=\PP(N_n=k) &= 2^k\sum_{i_1+\ldots+i_k=n}{1\over i_1(i_1+i_2)\cdots(i_1+\ldots+i_k)}\\
&\hspace{3cm}\times{i_1\cdots i_k\over (i_1+1)(i_1+i_2+1)\cdots(i_1+\ldots+i_k+1)},
\end{split}
\end{equation}
where the sum runs over all integers $i_1,\ldots,i_k\geq 1$ satisfying $i_1+\ldots+i_k=n$. It should be noted that the value $p_n^{(n)}={2^n\over n!(n+1)!}$ was earlier determined by B\'ar\'any, Roter, Steiger and Zhang \cite{BRSZ00} and $p_k^{(n)}=0$ for any $k\not\in\{1,\ldots,n\}$, $n\ge 1$; $p_k^{(0)}:=0$, $k\ge 1$; and $p_0^{(0)}:=1$. Moreover it is easy to check that $p_1^{(n)}={2\over n+1}$. For other values of $k$ the explicit expression \eqref{eq:Probabilities} is very hard to work with, but in the subsequent paper \cite{Buchta12} Buchta has shown that $p_{k}^{(n)}$, $n\in\NN$, $1\leq k\leq n$ satisfy the following recurrent relation
\begin{multline}\label{eq::reccforp}
\left(\frac{{n(n+1)}}{2}p^{(n)}_k - \frac{{(n-1)n}}{2}p^{(n-1)}_k\right)-\left(\frac{{(n-1)n}}{2}p^{(n-1)}_k- \frac{{(n-2)(n-1)}}{2}p^{(n-2)}_k \right)= p^{(n-1)}_{k-1},
\end{multline}
which in turn transferred to the recurrent relation for the moments $\EE[N_n^m]$, $m\ge 1$, $n\in\NN$ of $N_n$. Using this recurrent relation Buchta gave an explicit expression for the first few moments of $N_n$, namely
\begin{align}
\EE [N_n] &= {2\over 3}H_n + {1\over 3},\label{eq:ExpN}\\
\EE [N_n^2] &= {4\over 9}H_n^2 + {22\over 27}H_n+{4\over 9}H_n^{(2)}-{25\over 27}+{4\over 9(n+1)},\notag
\end{align}
where
\begin{equation}\label{eq:HarmonicNumbers}
H_n:=\sum_{i=1}^n{1\over i},\qquad H_n^{(k)}:=\sum_{i=1}^n{1\over i^k},\qquad k\ge 1,n\ge 1,
\end{equation}
are harmonic and generalised harmonic numbers, respectively. In principle, it would be possible to compute higher moments as well, but the formula become very involved. On the other hand the asymptotic relation as $n\to\infty$ takes a simple form \cite[Theorem 3]{Buchta12}
\begin{align*}
	\EE [N_n^k] = \Big({2\over 3}\log n\Big)^k + O((\log n)^{k-1}),\qquad k\ge 1.
\end{align*}

The relation \eqref{eq:Probabilities} turned out to be extremely important while studying the properties of the convex chain $T_n$. Thus, in \cite{GT21} it was used to obtain the recurrence relation for the probability generating function $G_n(t):=\sum_{k=1}^nt^{k}p_{k}^{(n)}$ of $N_n$, which is of the form 
\begin{equation}\label{eq::recforGnt}
G_n(t)=\left(\dfrac{2}{n(n+1)}t+\dfrac{2(n-1)}{n+1}\right)G_{n-1}(t)-\dfrac{(n-1)(n-2)}{n(n+1)}G_{n-2}(t),
\end{equation}
with $G_0(t)=1, G_1(t)=t$. The latter in particular implies that the sequence of monic polynomials $((p_n^{(n)})^{-1}G_n)_{n\in\NN}$ is the family of orthogonal polynomials and, thus, for any $n\in\NN$ the polynomial $G_n$ has only negative real roots. This leads to the surprising fact that $N_n$ has the same distribution as the sum of independent Bernoulli random variables, which allowed to derive a number of probabilistic limit theorems for $N_n$ including the quantitative central limit theorem as $n\to\infty$ with optimal speed of convergence $(\log(n))^{-1/2}$ in Kolmogorov distance (see \cite{GT21} for more details). The geometric interpretation of this stochastic representation remains unclear.

As an additional motivation to study random polygone $T_n$ let us point out that Groeneboom \cite{Groeneboom88} shows that if the container body $K$ is a polygone itself, then the vertices of a random polygone $K_n$ accumulate close to the corners of $K$ and that the behaviour in different corners is essentially independent. This motivates the study of $K_n$ in the neighbourhood of a single corner, which by affine invariance of $f_0(K_n)$ can be identified with the apex of a right angle $T$. Groeneboom used this idea to show that $f_0(K_n)$ satisfies  central limit theorem as $n\to\infty$. This approach was further successfully extended in \cite{GRT23} in order to obtained the quantitative version of the above central limit theorem with the speed of convergence of order $(\log(n))^{-1/2}$ in Kolmogorov distance. This speed of convergence is believed to be optimal and the proof relies on the results from \cite{GT21} for random convex chain.

\paragraph{Main results.}

In this article we aim to study the second functional of interest, namely the area $\vol(T_n)$ of a random convex chain $T_n$. We obtain the recurrence relation for $\mathbb{E}[\vol(T_n)^k]$ of the form
\[
\E [\vol(T_n)^k]=\dfrac{2n}{(n+k)(n+k+1)}\sum_{\ell=0}^k\sum_{m=0}^{n-1}\dfrac{{{n-1}\choose{m}}{{k}\choose{\ell}}}{{{n+k}\choose{m+\ell}}}2^{\ell-k}\E [\vol(T_m)^{\ell}],
\]
stated in Theorem \ref{thm:MomentRecursion}. Solving this recurrence relation for $n=1,2$ and any $k\ge 0$ we obtain exact formulas of the form 
\begin{align*}
    \EE [\vol(T_1)^k]&={1\over (k+1)(k+2)},\qquad k\ge 0,\\
        \EE [\vol(T_2)^k]&={4(k-3)\over2^{k}(k+1)(k+2)(k+3)}+{8H_{k+2}\over 2^{k}(k+1)(k+2)^2},\qquad k\ge 0.
    \end{align*}
    Moreover, using the relation between moments of random variables $N_n$ and $\vol(T_n)$ we obtain an exact formula for moments of $\vol(T_n)$ of any order
    \[
    \EE  [\vol(T_n)^k]=\sum_{m=1}^n\sum_{\ell=m}^{n+k}(-1)^{\ell+m}2^{m-k}{{n+k-m\choose k}{n+k\choose \ell}\over {n+k\choose k}}\sum_{1\leq j_1<\ldots<j_m\leq \ell}\prod_{i=1}^m{1\over j_i(j_i+1)},\qquad k\ge 0,
    \]
    stated in Theorem \ref{tm:MomentFormulaExact}. In particular for $k=1$, $k=2$ and any $n\ge 1$ we obtain (see Corollary \ref{cor:ExpectedVolume2})
    \begin{align*}
    \EE [\vol(T_n)]&=
    {1\over 2}-{1\over 6(n+1)}-{H_{n+1}\over 3(n+1)},\\
    \EE [\vol(T_n)^2]&=
    {1\over 4}-\frac{9n^2 + 53n + 72}{54 (n+1)(n+2)(n+3)}
-{(9n+8)H_{n+2}\over 27(n+1)(n+2)}+{H^2_{n+2}+H^{(2)}_{n+2}\over 9(n+1)(n+2)}.
    \end{align*} 
    It is possible to derive similar formulas for other values of $k$ in terms of harmonic and generalised harmonic numbers, but the formulas become very involved. As a consequence we also get an exact formula for the variance of $\vol(T_n)$, namely
    \begin{align*}
    {\rm var}(\vol(T_n))&=\frac{(7n^2+24n+14)H_{n+1}}{27(n+1)^2(n+2)^2}-\frac{H_{n+1}^2}{9(n+1)^2(n+2)}+ \frac{H_{n+1}^{(2)}}{9(n+1)(n+2)}\\
&\qquad\qquad
- \frac{55 n^4 + 391 n^3 + 962 n^2 + 956 n + 336}
{108(n+1)^2(n+2)^3(n+3)}
    \end{align*}
    In particular since $H_{n}=\log n+O(1)$ we get
    \[
    \EE[\vol(T_n)]={1\over 2}-{\log n\over 3n} + O(n^{-1}),\qquad {\rm var}(\vol(T_n))={7 \log n\over 27 n^2}+O(n^{-2}).
    \]
    
    As an intermediate step for the proof of Theorem \ref{thm:MomentRecursion} we showed that the generating function of a sequence $(G_n(t))_{n\in\NN}$ for any fixed $t\in\RR$ is given by some hypergeometric function (see Theorem \ref{tm:GenrationFunction}). {This result was independently and at the same time obtained by Besau and Th\"ale \cite{BT25}, where the authors used the methods of analytic combinatorics in order to study the quantitative central limit theorem and large deviation principle for $N_n$. Apart from Theorem \ref{tm:GenrationFunction} our results are disjoint from those of \cite{BT25}. We use the generating function of $(G_n(t))_{n\in\NN}$ in order to obtain an exact formula for the function $G_n(t)$ (see Theorem \ref{tm:FunctionG}), which in turn provides an alternative formula for probabilities 
        \[
    p_{k}^{(n)}=\sum_{\ell=k}^{n}(-1)^{\ell}{n\choose \ell}\sum_{1\leq j_1<\ldots<j_k\leq \ell}\prod_{i=1}^k{2\over j_i(j_i+1)},\qquad 0\leq k\leq n,
    \]
    as stated in Corollary \ref{cor:Probabilities}.

\paragraph{Structure of the paper.} The rest of the article is organized as follows. Section 2 contains some useful facts on random convex chains and hypergeometric functions. In Section 3 we consider the sequence $q_{n,k}:={{n+k}\choose {k} }\EE[2^k\vol(T_n)^2]$, $n,k\in\NN$ and study its properties. In particular, we derive an expression for its generating function in terms of hypergeometric function. In Section 4 we derive a number of results for random convex chain.

\section{Preliminaries}

\subsection{Random convex chain}

Given a set $A\subset \RR^d$ we denote by $\conv(A)$ its convex hull. Recall that by $T$ we denote the triangle with vertices $(0,1)$, $(0,0)$, $(1,0)$ in $\RR^2$ and by $T_n$ the convex hull of independent and uniformly distributed in $T$ points $X_1,\ldots,X_n$ together with points $(0,1)$ and $(1,0)$. Denote by 
\[
V_n:=\vol(T_n)/\vol(T)=2\vol(T_n),\qquad\qquad N_n:=f_0(T_n)-2.
\]
Using this notation we have that $p_{k}^{(n)}=\PP(N_n=k)$ (where $p_{k}^{(n)}$ are given by \eqref{eq:Probabilities}) and $G_n(t)=\sum_{k=1}^np_{k}^{(n)}t^k$ is the probability generating function of the random variable $N_n$.

In what follows we will be interested in $\EE V_n^k$ for $k\in\NN$. For convenience we introduce the normalized moment sequence, namely
\begin{equation}\label{eq:A}
q_{n,k}:={{n+k}\choose {k} }\E V_n^k,\qquad k,n\in\NN.
\end{equation}
Additionally we set $q_{n,0}:=\EE V_n^0=1$ for any $n\ge 0$ and $q_{0,k}:=0$ for any $k\ge 1$. 
Let
\begin{equation}\label{eq:DefF}
Q(x,y):=\sum_{n=0}^{\infty}\sum_{k=0}^{\infty}q_{n,k}x^ny^k
\end{equation}
be the generating function of the sequence $(q_{n,k})_{k,n\ge 0}$. Note that for any $k,n\in\NN$ since $T_n\subset T$ and, hence, $\vol(T_n)\leq 1/2$ almost surely we have $q_{n,k}\leq {n+k\choose k}$. This implies that for any $x,y\in\RR$ we have
\begin{align*}
\sum_{n=0}^{\infty}\sum_{k=0}^{\infty}{n+k\choose k}|x|^n|y|^k=\sum_{k=0}^{\infty}{|y|^k\over (1-|x|)^{k+1}}={1\over 1-|y|-|x|}<\infty\qquad \text{if}\quad |x|+|y|<1
\end{align*}
and the series converges absolutely. Thus, from \eqref{eq:A} it directly follows that
\begin{equation}\label{eq::Ankk!n!}
\EE V_n^k={n+k\choose k}^{-1} q_{n,k}=\dfrac{1}{(n+k)!}\dfrac{\partial^{n+k} }{\partial y^{k}\,\partial x^{n}} Q(x,y)\Big|_{x=y=0},
\end{equation}
and knowing the function $Q(x,y)$ provides an access to the values $\EE V_n^k$.

\subsection{Hypergeometric function}

Given $a\in \CC$ and $n\in \NN$ the rising factorial of $a$ is denoted by $(a)_n=a(a+1)\cdots (a+n-1)$ and we set $(a)_0:=1$. Note that $(a)_n={\Gamma (a+n)\over \Gamma(a)}$ if $a\neq 0, -1,-2,\ldots$. 

In what follows we will use the Gauss hypergeometric function $F(a,b;c;z)$, $a,b,c\in\CC$, $c\neq 0,-1,-2,\ldots$, which is defined by Gauss series
\begin{equation}\label{eq:HypergeometricFuncDef}
F(a,b;c;z)=\sum_{n=0}^{\infty}{(a)_n(b)_n\over n!(c)_n}z^n,
\end{equation}
on the unit disc $|z|<1$ and by analytic continuation elsewhere. A few useful properties of function $F(a,b;c;z)$ are the following \cite[Equation 15.5.1]{NIST}
\begin{equation}\label{eq:HypergeometricFuncDeriv}
    {{\rm d}^n\over {\rm d}z^n}F(a,b;c;z)={(a)_n(b)_n\over (c)_n}F(a+n,b+n;c+n;z)
\end{equation}
and \cite[Equation 15.8.1]{NIST}
\begin{equation}\label{eq:HypergeometricFuncEquiv}
    F(a,b;c;z)=F(b,a;c;z)=(1-z)^{-a}F\Big(a,c-b;c;{z\over z-1}\Big)=(1-z)^{-b}F\Big(c-a,b;c;{z\over z-1}\Big).
\end{equation}
For more details and further properties of hypergeometric functions we refer to \cite[Section 15]{NIST}.

\section{Properties of the normalized moment sequence $(q_{n,k})_{n,k\ge 0}$}

In this section we study the properties of the sequence $(q_{n,k})_{n,k\ge 0}$ and in particular its generating function $Q$. We start with a few simple properties which will be useful in the sequel, followed by recurrence relation for $(q_{n,k})_{n,k\ge 0}$ and representation of generating function $Q$ in terms of hypergeometric function and probability generating functions $(G_n)_{n\ge 0}$.

\subsection{Some useful identities}

The first proposition provides a relation between the normalised moments $(q_{n,k})_{n,k\ge 0}$ and the probabilities $(p_{k}^{(n)})_{n\ge 0, 0\leq k\leq n}$.

\begin{proposition}\label{lm:RecurrenceAP}
For any $n, k\ge 0$ we have
\begin{equation}\label{eq::Anksumpi}
q_{n,k}=\sum_{i=0}^n {{n+k-i}\choose k} p_i^{(n+k)},
\end{equation}
where $q_{n,k}$ is defined by \eqref{eq:A} and $p_{k}^{(n)}$ are given in \eqref{eq:Probabilities}.
\end{proposition}

\begin{remark}
    This kind of relation holds for many models of random polytopes, for example the identical equality relating the moments of the volume of random polytope $K_n$ with $(\PP(f_0(K_n)=k))_{0\leq k\leq n}$ follows from \cite[Theorem 1, Theorem 2]{Buch05}. 
\end{remark}

\begin{proof}
The proof of \eqref{eq::Anksumpi} is in fact identical to the proof of Theorem 1 and Theorem 2 in \cite{Buch05}. For readers convenience we present some details here.

First we note, that \eqref{eq::Anksumpi} is equivalent to
\begin{equation}\label{eq:16.07.24_1}
{n+k\choose k}{\EE [\vol(T_n)^k]\over (\vol(T))^k}=\EE{n+k-N_{n+k}\choose k}.
\end{equation}
In order to prove \eqref{eq:16.07.24_1} we proceed as follows. Let $Y_1,\ldots, Y_{n+k}$ be independent random points uniformly distributed in $T$ and let $P_{n,k}$ be the number of subsets $\{Y_{i_1},\ldots, Y_{i_k}\}\subset \{Y_1,\ldots,Y_{n+k}\}$ of distinct points such that
\[
Y_{i_1},\ldots, Y_{i_k}\in \conv(\{Y_1,\ldots,Y_{n+k}, (0,1), (1,0)\}\setminus \{Y_{i_1},\ldots, Y_{i_k}\}).
\]
Note that the above condition is equivalent to the fact that none of the points $Y_{i_1},\ldots, Y_{i_k}$ is a vertex of $\conv(Y_1, \ldots, Y_{n+k},(0,1),(1,0))$. Further since $N_{n+k}$ has the same distribution as $f_0(\conv(Y_1, \ldots, Y_{n+k},(0,1),(1,0)))$ we have that 
\begin{equation}\label{eq:16.07.24_2}
    \EE P_{n,k}=\EE {n+k-N_{n+k}\choose k}.
\end{equation}
On the other hand since $Y_1,\ldots, Y_{n+k}$ are independent and identically distributed and $T_n$ has the same distribution as $\conv(Y_1,\ldots,Y_n,(0,1),(1,0))$ we have
\begin{align*}
    \EE P_{n,k}&=\EE\sum_{1\leq i_1<\ldots<i_k\leq n+k}{\bf 1}\{Y_{i_1},\ldots, Y_{i_k}\in \conv(\{Y_1,\ldots,Y_{n+k}, (0,1), (1,0)\}\setminus \{Y_{i_1},\ldots, Y_{i_k}\})\}\\
    &={n+k\choose k}\PP(Y_{n+1},\ldots, Y_{n+k}\in T_n)={n+k\choose k}{\EE [\vol(T_n)^k]\over (\vol T)^k}.
\end{align*}
Combining this with \eqref{eq:16.07.24_2} proves \eqref{eq:16.07.24_1} and, hence, \eqref{eq::Anksumpi}.
\end{proof}

As a consequence of Proposition \ref{lm:RecurrenceAP} the next result provides a connection between the subsequence $(q_{n-k,k})_{0\leq k\leq n}$ and the probability generating function $G_n$ of the number of vertices of the random convex chain.

\begin{proposition}\label{prop:UsefulFormula1}
    For any $n\ge 0$ and $x,y\in\RR$ we have 
    \[
    \sum_{k=0}^n q_{n-k,k}x^{n-k}y^k=(x+y)^nG_n\Big({x\over x+y}\Big).
    \]
\end{proposition}

\begin{proof}
    By \eqref{eq::Anksumpi} we get
    \begin{align*}
        \sum_{k=0}^n q_{n-k,k}x^{n-k}y^k&=\sum_{k=0}^n\sum_{i=0}^{n-k} {{n-i}\choose k} p_i^{(n)}x^{n-k}y^k\\
        &=\sum_{i=0}^{n} p_i^{(n)}x^i\sum_{k=0}^{n-i}{{n-i}\choose k}x^{n-i-k}y^k\\
        &=\sum_{i=0}^{n} (x+y)^{n-i}x^ip_i^{(n)}=(x+y)^nG_n\Big({x\over x+y}\Big).
    \end{align*}
\end{proof}

\subsection{Recurrence relation}

Next theorem provides a recurrence relation for the sequence $(q_{n,k})_{n,k\ge 0}$, similar in spirit to the relation \eqref{eq::reccforp} for the probabilities $(p_{k}^{(n)})_{n\ge 0, 0\leq k\leq n}$.

\begin{theorem}
For any $n\ge 1$, $k\ge 2$ we have
\begin{multline}\label{eq::recAnk}
q_{n,k}=\dfrac{(n+k+1)(n+k+2)}{2} q_{n+1,k}-{(n+k)(n+k+1)} \left(q_{n,k}+q_{n+1,k-1}\right)\\
+\dfrac{(n+k-1)(n+k)}{2}\left(q_{n-1,k}+2q_{n,k-1}+q_{n+1,k-2}\right),
\end{multline}
where $q_{n,k}$ is defined by \eqref{eq:A}. Moreover we have $q_{n,0}=1$ for $n\ge 0$; $q_{0,k}=0$  for $k\ge 1$; and $q_{n,1}=n+{2\over 3}-{2\over 3}H_{n+1}$, $q_{1,k}={2\over k+2}$.
\end{theorem}

\begin{proof}
By \eqref{eq::Anksumpi} we have the following equalities:
\begin{align}\label{eq::Ankdfifferentnk}
q_{n+1,k-1}&=\sum_{i=1}^{n+1} {{n+k-i}\choose {k-1}} p_i^{(n+k)}\\
q_{n,k-1}&=\sum_{i=1}^{n} {{n+k-1-i}\choose {k-1}} p_i^{(n+k-1)}\\
q_{n-1,k}&=\sum_{i=1}^{n-1} {{n+k-1-i}\choose {k}} p_i^{(n+k-1)}
\end{align}
Moreover, using  \eqref{eq::Anksumpi}  and the recursive relation for binomial coefficients
\begin{equation}\label{eq::chooserec}
{n \choose k}={{n-1} \choose {k-1}}+{{n-1} \choose k}={{n-1} \choose {k-1}}+{{n-2} \choose k-1}+{{n-2} \choose k}
\end{equation}
we get 
\begin{equation}
\begin{split}
q_{n+1,k}&=\sum_{i=1}^{n+1} {{n+k+1-i}\choose {k}} p_i^{(n+k+1)}\\
&=\sum_{i=1}^{n+1} {{n+k-i}\choose {k-1}} p_i^{(n+k+1)}+\sum_{i=1}^{n} {{n+k-i-1}\choose {k-1}} p_i^{(n+k+1)}\\
&\hspace{7cm}+\sum_{i=1}^{n-1} {{n+k-i-1}\choose {k}} p_i^{(n+k+1)},
\end{split}
\end{equation}
where we additionally used the fact that ${n\choose k}=0$ whenever $k>n$.
Similarly, using again \eqref{eq::Anksumpi}, \eqref{eq::chooserec}  and ${n\choose k}=0$, whenever $k>n$, we get
\begin{equation}
\begin{split}
q_{n,k}&=\sum_{i=1}^{n-1} {{n+k-i-1}\choose {k}} p_i^{(n+k)}+\sum_{i=1}^{n} {{n+k-i-1}\choose {k-1}} p_i^{(n+k)}\\
\end{split}
\end{equation}
Further, we notice that due to \eqref{eq::Anksumpi} and \eqref{eq::chooserec} we have 
\begin{equation}\label{eq::sumAA}
\begin{split}
q_{n+1,k-2}+q_{n,k-1}&=\sum_{i=1}^{n+1} {{n+k-i-1}\choose {k-2}} p_i^{(n+k-1)}+\sum_{i=1}^{n} {{n+k-i-1}\choose {k-1}} p_i^{(n+k-1)}\\
&=\sum_{i=1}^{n+1} {{n+k-i}\choose {k-1}} p_i^{(n+k-1)}.
\end{split}
\end{equation}
Finally, using  \eqref{eq::Ankdfifferentnk} -- \eqref{eq::sumAA} and the equation \eqref{eq::reccforp} we get, that right hand side $R$ of \eqref{eq::recAnk} is equal to
\begin{align*}
R &= \sum_{i=1}^{n-1} {{n+k-1-i}\choose {k}} p_{i-1}^{(n+k)}+\sum_{i=1}^{n} {{n+k-1-i}\choose {k-1}} p_{i-1}^{(n+k)}+\sum_{i=1}^{n+1} {{n+k-i}\choose {k-1}} p_{i-1}^{(n+k)}\\
&=\sum_{i=1}^{n+1} {{n+k-i+1}\choose {k}} p_{i-1}^{(n+k)},
\end{align*}
where in the last line we again have used \eqref{eq::chooserec} and ${n\choose k}=0$, whenever $k>n$. Now we make the change of numeration $j:=i-1$ in the last equation and, using that $p^{(j)}_0=0$ for $j\ge 1$ we obtain
$$
R=\sum_{j=1}^{n} {{n+k-j}\choose {k}} p_{j}^{(n+k)}=q_{n,k},
$$
which finishes the proof of recurrence relation.

The equality 
\[
q_{n,1}=n+{2\over 3}-{2\over 3}H_{n+1}
\]
follows from $q_{n,1}=\EE[n+1-N_{n+1}]$, which holds due to \eqref{eq:16.07.24_1}, and \eqref{eq:ExpN}.
\end{proof}

\subsection{Generating function}

In this subsection we study properties of the generating function $Q(x,y)$ of the sequence $(q_{n,k})_{n,k\ge 0}$. Let us point out that in contrast to the recurrence relation in the previous section the generating function $Q(x,y)$ encodes the information on the whole sequence $(q_{n,k})_{n,k\ge 0}$. 

Motivated by Proposition \ref{prop:UsefulFormula1} we start by showing the following identity relating the generating function $Q(x,y)$ with the sequence of probability generating functions $(G_n)_{n\ge 0}$.

\begin{proposition}\label{proposition6}
For any $t,s\in\RR$ with $|x|+|y|<1$ we have
\[
Q(x,y)=\sum_{\ell=0}^{\infty}(x+y)^{\ell}G_{\ell}\Big({x\over x+y}\Big).
\]
\end{proposition}

\begin{proof}
We start by noting that according to \eqref{eq:DefF} and setting $\ell=n+k$ we get
\[
Q(x,y)=\sum_{n=0}^{\infty}\sum_{k=0}^{\infty}q_{n,k}x^ny^k=\sum_{\ell=0}^{\infty}\sum_{k=0}^{\ell}q_{\ell-k,k}x^{\ell-k}y^k.
\]
The statement now follows by Proposition \ref{prop:UsefulFormula1}.
\end{proof}

Further given $s\in\RR$ let 
\begin{equation}\label{eq:LfunctionDef}
L(s;z):=\sum_{\ell=0}^{\infty}z^{\ell}G_{\ell}(s)
\end{equation}
be the generating function of the sequence $(G_{\ell}(s))_{\ell\ge 0}$. We note that by Proposition \ref{prop:UsefulFormula1} we have 
\[
Q(x,y)=L\Big({x\over x+y};x+y\Big),\qquad L(s;z)=Q(s z,(1-s)z).
\]
The main result of this subsection is the following theorem.

\begin{theorem}\label{tm:GenrationFunction}
For any $s\in \RR$ and $|z|<(|s|+|1-s|)^{-1}$ we have that the series in \eqref{eq:LfunctionDef} converges absolutely and
\begin{align*}
L(s;z)&=(1-z)^{\beta}F(\beta,\beta+1;2;z)=F\Big(\beta,1-\beta;2;{z\over z-1}\Big),
\end{align*}
where $\beta={1\over 2}+{1\over 2}\sqrt{1+8s}$ and $F$ is the hypergeometric function defined by \eqref{eq:HypergeometricFuncDef}.
\end{theorem}

\begin{remark}
    Let us consider one direct application of the above theorem. Consider 
    \begin{align*}
        \lim_{s\to 0}L(-s^{-1};-s z)&=\lim_{s\to 0}\sum_{\ell=0}^{\infty}z^{\ell}(-s)^{\ell}G_{\ell}(-s^{-1})=\sum_{\ell=0}^{\infty}z^{\ell}p_{\ell}^{(\ell)}.
    \end{align*}
    On the other hand by Theorem \ref{tm:GenrationFunction} for any $z\in\CC$ with $|z|<{1\over 2}$ we get
    \begin{align*}
    \lim_{s\to 0}L(-s^{-1};-s z)&=\lim_{s\to 0}F\Big({1\over 2}+{1\over 2}\sqrt{1-{8\over s}}, {1\over 2}-{1\over 2}\sqrt{1-{8\over s}};2;{sz\over 1+sz}\Big).
    \end{align*}
    Further, since $({1\over 2}+{1\over 2}\sqrt{1-{8\over s}})({1\over 2}-{1\over 2}\sqrt{1-{8\over s}})={2\over s}$, ${sz\over 1+sz}=sz(1+o_s(1))$ and $F(a,b;c;z)$ is continuous in $a,b,z$ by \cite[Equation 10.39.10]{NIST}  with $x=2\sqrt{2z}$ we get
    \[
    \lim_{a,b\to\infty}F\Big(a,b;2;{x^2\over 4ab}\Big)=\Big({x\over 2}\Big)^{-1}I_1(x),
    \]
    where
\[
I_{\nu}(z)=\sum_{k=0}^{\infty}{(z/2)^{2k+\nu}\over k!\Gamma(k+1+\nu)},\qquad z\in\CC\setminus (-\infty,0].
\]
is the modified Bessel function. Thus, we conclude that
    \[
    \lim_{s\to 0}L(-s^{-1};-s z)={I_1(2\sqrt{2z})\over \sqrt{2z}}=\sum_{\ell=0}^{\infty}{(2z)^\ell\over \ell!(\ell+1)!}=\sum_{\ell=0}^{\infty}z^{\ell}p_{\ell}^{(\ell)},
    \]
    which implies the known formula $p_{\ell}^{(\ell)}={2^{\ell}\over \ell!(\ell+1)!}$.
\end{remark}

In order to prove Theorem \ref{tm:GenrationFunction} we start by showing a few auxiliary statements, which are of independent interest.

\begin{lemma}\label{lm:QtildeLemma}
Let 
\[
\widetilde Q(x,y):=\sum_{k=0}^\infty \sum_{n=0}^\infty (n+k)(n+k+1)q_{n,k} x^n y^k,
\]
then for ant $x,y\in\RR$ with $|x|+|y|<1$ the above series is absolutely convergent and the following holds
\begin{equation}\label{eq::widePtssumGl}
\widetilde Q(x,y)=\sum_{\ell=0}^\infty \ell(\ell+1)(x+y)^{\ell} G_{\ell}\Big(\dfrac{x}{x+y}\Big)=\dfrac{2 x\, Q(x,y)}{(1-x-y)^2}.
\end{equation}
\end{lemma}

\begin{proof}
Setting $\ell=n+k$ and applying Proposition \ref{prop:UsefulFormula1} we get
\begin{align*}
\widetilde Q(x,y)&=\sum_{\ell=0}^\infty \ell(\ell+1)\sum_{k=0}^{\ell} q_{\ell-k,k} x^{\ell-k}y^k
=\sum_{\ell=1}^\infty \ell(\ell+1)(x+y)^{\ell} G_{\ell}\Big(\dfrac{x}{x+y}\Big),
\end{align*}
which gives the first equality in \eqref{eq::widePtssumGl}. In particular this implies that 
\[
|\widetilde Q(x,y)|\leq \sum_{\ell=1}^{\infty}\ell(\ell+1)\sum_{i=0}^{\ell}p_i^{(\ell)}|x|^i(|x|+|y|)^{\ell-i}\leq \sum_{\ell=1}^{\infty}\ell(\ell+1)(|x|+|y|)^{\ell}<\infty
\]
for $|x|+|y|<1$. Further since $G_1(y)=y$ we obtain
\begin{align*}
\widetilde Q(x,y)&=2x+\sum_{\ell=2}^\infty \ell(\ell+1)(x+y)^{\ell} G_{\ell}\left(\dfrac{y}{x+y}\right),
\end{align*}
and using recurrence relation \eqref{eq::recforGnt} for $G_{\ell}(y)$ we get
\begin{align*}
\widetilde Q(x,y)
&=2x+\sum_{\ell=2}^\infty \ell(\ell+1)(x+y)^{\ell} \left(\dfrac{2x}{\ell(\ell+1)(x+y)}+\dfrac{2(\ell-1)}{(\ell+1)}\right)G_{\ell-1}\Big(\dfrac{x}{x+y}\Big)\\
&\hspace{5cm}-\sum_{\ell=2}^\infty \ell(\ell+1)(x+y)^{\ell} \dfrac{(\ell-1)(\ell-2)}{\ell(\ell+1)}G_{\ell-2}\Big(\dfrac{x}{x+y}\Big)\\
&=2x +2 x\sum_{\ell=2}^\infty (x+y)^{\ell-1}G_{\ell-1}\Big(\dfrac{x}{x+y}\Big)+2\sum_{\ell=2}^\infty \ell(\ell-1)(x+y)^{\ell} G_{\ell-1}\Big(\dfrac{x}{x+y}\Big)\\
&\hspace{5cm}-\sum_{\ell=2}^\infty {(\ell-1)(\ell-2)}(x+y)^{\ell} G_{\ell-2}\Big(\dfrac{x}{x+y}\Big)\\
&=2x Q(x,y)+2(x+y)\widetilde Q(x,y)-(x+y)^2 \widetilde Q(x,y),
\end{align*}
where the last line is true due to Proposition \ref{proposition6}.

Therefore, we conclude that
\begin{align*}
\widetilde Q(x,y)\left[1-2(x+y)+(x+y)^2\right]=\widetilde Q(x,y)(x+y-1)^2=2 x Q(x,y),
\end{align*}
and the second equality in \eqref{eq::widePtssumGl} follows.
\end{proof}

Next we show that the function $Q(x,y)$ satisfies a certain partial differential equation of parabolic type with given initial conditions.

\begin{proposition}
The function $Q(x,y)$ satisfy the following parabolic type partial differential equation
\begin{equation}\label{eq::diffPts}
x^2\dfrac{\partial^2 Q}{\partial x^2}+y^2\dfrac{\partial^2 Q}{\partial y^2}+2xy\dfrac{\partial^2 Q}{\partial x\partial y}+2x \dfrac{\partial Q}{\partial x}+2y \dfrac{\partial Q}{\partial y}-\dfrac{2x}{(1-x-y)^2}Q=0, \qquad x,y\in\RR,\, |x|+|y|<1,
\end{equation}
with initial conditions 
\begin{align*}
&Q(0,y)=1,\qquad &\text{ and }&\qquad Q(x,0)=\dfrac{1}{1-x},\\ 
&{\partial \over \partial x} Q(0,y)={2\over y^2}\log\Big({e^{-y}\over 1-y}\Big),\quad {\partial \over \partial y} Q(0,y)=0,\qquad&\text{ and }&\qquad{\partial \over \partial x} Q(x,0)={1\over (1-x)^2}.
\end{align*}
\end{proposition}

\begin{proof}
Let us denote 
\[
F(x,y)=(x+y)Q(x,y)=\sum_{\ell=0}^{\infty}(x+y)^{\ell+1}G_{\ell}\Big(\dfrac{x}{x+y}\Big)
\]
as follows by Proposition \ref{prop:UsefulFormula1}. We would like to compute the partial derivatives of $F$ with respect to $x$ and $y$ up to order $2$. For this we first need to ensure that we may exchange the summation and differentiation. Recall that for any $\ell\ge 0$ we have that 
\[
f_{\ell}(x,y):=(x+y)^{\ell+1}G_{\ell}\Big({x\over x+y}\Big)=\sum_{i=0}^{\ell}p_{i}^{(\ell)}x^i(x+y)^{\ell-i+1},
\]
is a polynome in $x$ and $y$ and, hence, infinitely differentiable. Moreover, for any $\varepsilon >0$ and any $x,y\in\RR$ with $|x|+|y|\leq 1-\varepsilon$ we have that 
\[
\Big|{\partial f_{\ell}(x,y)\over \partial y}\Big|\leq \sum_{i=0}^{\ell}(\ell-i+1)p_i^{(\ell)}|x|^i(|x|+|y|)^{\ell-i}\leq (\ell+1)(|x|+|y|)^{\ell}\sum_{i=0}^{\ell}p_i^{(\ell)}\leq (\ell+1)(1-\varepsilon)^{\ell},
\]
since $\sum_{i=0}^{\ell}p_i^{(\ell)}=1$. Similarly we get
\begin{align*}
&\Big|{\partial f_{\ell}(x,y)\over \partial x}\Big|\leq (\ell+1)(1-\varepsilon)^{\ell}, \qquad &\Big|{\partial^2 f_{\ell}(x,y)\over \partial x^2}\Big|\leq 4(\ell+1)\ell(1-\varepsilon)^{\ell-1},\\
&\Big|{\partial^2 f_{\ell}(x,y)\over \partial y^2}\Big|\leq (\ell+1)\ell(1-\varepsilon)^{\ell-1},\qquad &\Big|{\partial^2 f_{\ell}(x,y)\over \partial x\,\partial y}\Big|\leq 2(\ell+1)\ell(1-\varepsilon)^{\ell-1}.
\end{align*}
Since the series $\sum_{\ell=0}^{\infty}(\ell+1)(1-\varepsilon)^{\ell}$ and $\sum_{\ell=0}^{\infty}(\ell+1)\ell(1-\varepsilon)^{\ell-1}$ converge for any $\varepsilon>0$ by \cite[Theorem 6.28]{Klenke} we may exchange differentiation with respect to $x$ and $y$ and summation with respect to $\ell$ for $|x|+|y|<1$. Thus, we get
\begin{align*}
\dfrac{\partial F}{\partial y}&=\sum_{\ell=0}^{\infty}\Big[(\ell+1)(x+y)^{\ell}G_{\ell}\Big(\dfrac{x}{x+y}\Big)-(x+y)^{\ell-1}G'_{\ell}\Big(\dfrac{x}{x+y}\Big)x\Big],\\
\dfrac{\partial F}{\partial x}&=\sum_{\ell=0}^{\infty}\Big[(\ell+1)(x+y)^{\ell}G_{\ell}\Big(\dfrac{x}{x+y}\Big)+(x+y)^{\ell-1}G'_{\ell}\Big(\dfrac{x}{x+y}\Big)y\Big].
\end{align*}
and for the second derivatives we obtain
\begin{align*}
\dfrac{\partial^2 F}{\partial x\,\partial y}&=\sum_{\ell=0}^{\infty}\Big[(\ell+1)\ell(x+y)^{\ell-1}G_{\ell}\Big(\dfrac{x}{x+y}\Big)+(y-x)\ell(x+y)^{\ell-2}G'_{\ell}\Big(\dfrac{x}{x+y}\Big)-xy(x+y)^{\ell-3}G''_{\ell}\Big(\dfrac{x}{x+y}\Big)\Big],\\
\dfrac{\partial^2 F}{\partial y^2}&=\sum_{\ell=0}^{\infty}\Big[(\ell+1)\ell(x+y)^{\ell-1}G_{\ell}\Big(\dfrac{x}{x+y}\Big)-2x\ell(x+y)^{\ell-2}G'_\ell\Big(\dfrac{x}{x+y}\Big)+x^2(x+y)^{\ell-3}G''_{\ell}\Big(\dfrac{x}{x+y}\Big)\Big],\\
\dfrac{\partial^2 F}{\partial x^2}&=\sum_{\ell=0}^{\infty}\Big[(\ell+1)\ell(x+y)^{\ell-1}G_{\ell}\Big(\dfrac{x}{x+y}\Big)+2y\ell(x+y)^{\ell-2}G'_{\ell}\Big(\dfrac{x}{x+y}\Big)+y^2(x+y)^{\ell-3}G''_{\ell}\Big(\dfrac{x}{x+y}\Big)\Big].
\end{align*}
Now one can directly check, that due to \eqref{eq::widePtssumGl} the following holds
\begin{equation*}\label{eq:16.07.24_3}
\dfrac{x^2}{x+y}\dfrac{\partial^2 F}{\partial x^2}+\dfrac{y^2}{x+y}\dfrac{\partial^2 F}{\partial y^2}+\dfrac{2xy}{x+y}\dfrac{\partial^2 F}{\partial x\partial y}=\widetilde Q=\dfrac{2x\, Q}{(1-x-y)^2}.
\end{equation*}

Finally, since
\[
\dfrac{\partial^2 F}{\partial x^2}=(x+y)\dfrac{\partial^2 Q}{\partial x^2}+2\dfrac{\partial Q}{\partial x},\quad \dfrac{\partial^2 F}{\partial y^2}=(x+y)\dfrac{\partial^2 Q}{\partial y^2}+2\dfrac{\partial Q}{\partial y},\quad \dfrac{\partial^2 F}{\partial x\,\partial y}=(x+y)\dfrac{\partial^2 Q}{\partial x\,\partial y}+\dfrac{\partial Q}{\partial x}+\dfrac{\partial Q}{\partial y},
\]
the proof follows.

For the initial conditions we note that $q_{n,0}=1$ for $n\ge 0$ and $q_{0,k}=0$ for $k\ge 1$. Thus,
\[
Q(0,y)=\sum_{k=0}^{\infty}q_{0,k}y^k=q_{0,0}=1,\qquad Q(x,0)=\sum_{n=0}^{\infty}q_{n,0}x^n=\sum_{n=0}^{\infty}x^n={1\over 1-x}, \quad |x|<1.
\]
Moreover, according to Proposition \ref{prop:UsefulFormula1} and using the same arguments as above we get
\begin{align*}
\dfrac{\partial Q}{\partial y}&=\sum_{\ell=0}^{\infty}\Big[\ell(x+y)^{\ell-1}G_{\ell}\Big(\dfrac{x}{x+y}\Big)-(x+y)^{\ell-2}G'_{\ell}\Big(\dfrac{x}{x+y}\Big)x\Big],\\
\dfrac{\partial Q}{\partial x}&=\sum_{\ell=0}^{\infty}\Big[\ell(x+y)^{\ell-1}G_{\ell}\Big(\dfrac{x}{x+y}\Big)+(x+y)^{\ell-2}G'_{\ell}\left(\dfrac{x}
{x+y}\right)y\Big].
\end{align*}
Thus, using the fact that $G_{\ell}(1)=1$ for any $\ell\ge 0$, $G_{\ell}(0)=0$ if $\ell\ge 1$, and that $G_{\ell}'(0)=p_1^{(\ell)}={2\over \ell+1}$ as follows by \eqref{eq:Probabilities}, we get
\begin{align*}
    {\partial \over \partial x} Q(x,0)&=\sum_{\ell=1}^{\infty}\ell x^{\ell-1}G_{\ell}(1)={1\over (1-x)^2},\qquad |x|<1,\\
    {\partial \over \partial x} Q(0,y)&=\sum_{\ell=1}^{\infty}\Big[\ell y^{\ell-1}G_{\ell}(0)+y^{\ell-1}G_{\ell}'(0)\Big]={2\over y^2}\sum_{\ell=2}^{\infty}{y^{\ell}\over \ell}={2\over y^2}\log\Big({e^{-y}\over 1-y}\Big),\qquad |y|<1,\\
    {\partial \over \partial y} Q(0,y)&=\sum_{\ell=1}^{\infty}\ell y^{\ell-1}G_{\ell}(0)=0,\qquad |y|<1.
\end{align*}
\end{proof}

We are now in the position to proof Theorem \ref{tm:GenrationFunction}.

\begin{proof}[Proof of Theorem \ref{tm:GenrationFunction}]
First we note that for any $s\in\RR$ we have
\[
|G_{\ell}(s)|\leq \sum_{k=1}^{\ell}|s|^kp_{k}^{(\ell)}\leq \max (1,|s|)^{\ell}\leq (|s|+|1-s|)^{\ell},
\]
and, hence,
\[
\sum_{\ell=0}^{\infty}|z|^{\ell}|G_{\ell}(s)|\leq \sum_{\ell=0}^{\infty}((|s|+|1-s|)z)^{\ell}<\infty,
\]
for $|z|<(|s|+|1-s|)^{-1}$ and the series converges absolutely.

Further we note that $L(s;z)=Q(sz,(1-s)z)$. Thus, we introduce change of variables $z=x+y$, $s=x/(x+y)$ and with this change of variables equation \eqref{eq::diffPts} becomes a second-order ordinary linear differential equation for the function $f(z):=Q(sz, (1-s)z)=L(s;z)$, $s\in\RR$ namely
\[
zf_{zz}''+2 f_{z}'-\dfrac{2s}{(z-1)^2}f=0,\qquad |z|<{1\over |s|+|1-s|},
\]
with initial conditions $f(0)=1$ and $f'_z(0)=s$. Further let $\beta={1\over 2}+{1\over 2}\sqrt{1+8s}$. Then, by setting $w(z):=(1-z)^{-\beta}f(z)$ we conclude that $w$ satisfies the following differential equation 
$$
z(z-1)w''_{zz}+(z(2\beta+2)-2)w_z'+\beta(\beta+1)w=0,\qquad |z|<{1\over |s|+|1-s|}\leq 1,
$$
with initial conditions $w(0)=1$, $w'_z(0)={\beta(\beta+1)\over 2}$. The above differential equation is Gauss hypergeometric equation, whose particular solution is given by Gauss hypergeometric function
\[
w_0(z)=F(\beta,\beta+1;2;z),
\]
defined by \eqref{eq:HypergeometricFuncDef}. If $\beta\neq 1$ then the fundamental solution at $z=0$ has the form \cite[Section 15.10]{NIST}
\begin{align*}
w(z)&=c_1 F(\beta,\beta+1;2;z)+ c_2 \Big(F(\beta,\beta+1;2;z)\log(z)\\
&\quad+{1\over \beta(\beta-1)}{1\over z}+\sum_{k=0}^{\infty}{(\beta)_k(\beta+1)_k\over (2)_kk!}z^k(\psi(\beta+k)+\psi(\beta+1+k)-\psi(1+k)-\psi(2+k))\Big),
\end{align*}
where $\psi(z)=\Gamma'(z)/\Gamma(z)$, $z\neq 0,-1,-2,\ldots$ and $c_1,c_2\in \RR$. Note that according to initial conditions and since $F(\beta,\beta+1;2;0)=1$ we have
\[
w(0)=c_1 + c_2(\psi(\beta)+\psi(\beta+1)-\psi(1)-\psi(2)) +c_2\lim_{z\to 0}\Big(\log (z)+{1\over \beta(\beta-1)}{1\over z}\Big)=1.
\]
Since $\log (z)+{1\over \beta(\beta-1)}{1\over z}$ is not bounded at $z=0$ we conclude that $c_2=0$ and $c_1=1$ and, thus, $w(z)=F(\beta,\beta+1;2;z)$ is the only solution in this case. If $\beta=1$ the fundamental solution at $z=0$ is of the form \cite[Section 15.10]{NIST}
\[
w(z)=c_1 F(\beta,\beta+1;2;z)+c_2z^{-1}\lim_{c\to 0}\Big(\lim_{a\to 0}F(a,1;c;z)\Big),\qquad c_1,c_2\in\RR.
\]
Further note that according to \eqref{eq:HypergeometricFuncDef} we have
\begin{align*}
    \lim_{c\to 0}\Big(\lim_{a\to 0}F(a,1;c;z)\Big)=1
\end{align*}
for any $z\in\CC$ with $|z|<1$. Hence due to initial condition $w(0)=1$ and since the function $z^{-1}$ is not bounded at $z=0$ we conclude that $c_2=0$ and $c_1=1$. Combining this together leads to
\[
L(s;z)=f(z)=(1-z)^{\beta}w(z)=(1-z)^{\beta}F(\beta,\beta+1;2;z)=F\Big(\beta,1-\beta;2;{z\over z-1}\Big),
\]
where the last equality follows from \eqref{eq:HypergeometricFuncEquiv}. This finishes the proof of the first part of the theorem.

\end{proof}

\section{Identities for the random convex chain}

In this section we present a number of results for the random convex chain, including the recurrence relation and some exact formulas for the expected moments $\EE V_n^k$, $n,k\ge 0$ and the exact formulas for the probability generating functions $G_n$, $n\ge 1$. In order to not break the flow of representation the proofs will be given at the end of the section.

\begin{theorem}\label{thm:MomentRecursion}
For any $n\in\NN$ and $k\ge 0$ the following recurrent equality holds
\begin{equation}\label{eq::EVsum}
\E V_n^k=\dfrac{2n}{(n+k)(n+k+1)}\sum_{\ell=0}^k\sum_{m=0}^{n-1}\dfrac{{{n-1}\choose{m}}{{k}\choose{\ell}}}{{{n+k}\choose{m+\ell}}}\E V_{m}^{\ell},
\end{equation}
where we recall that $\EE V_0^k:=0$ for $k\ge 1$ and $\EE V_0^0:=1$.
\end{theorem}

Recurrence relation from Theorem \ref{thm:MomentRecursion} provides a way to compute all values of $(\EE V_{n}^k)_{k,n\ge 0}$, but the formulas become very involved. Moreover in order to compute $\EE V_n^k$ for given $k$ and all $n$ one would need to know the values $\EE V_m^k$ for all $m<n$, which makes the recurrence with respect to the number of points $n$ more complicated to apply. The recurrence with respect to degree $k$ is easier since $\EE V_n^k$ can be computed as soon as the values $\EE V_m^{\ell}$ for $m<n$ and $\ell<k$ are given. Below we present a few simple exact formulas for small values of  $k$.

\begin{corollary}\label{cor:ExpectedVolume1}
    For any $k\ge 0$ we get
    \begin{align*}
        \EE V_1^k&={2\over (k+1)(k+2)},\\
        \EE V_2^k&={4(k-3)\over(k+1)(k+2)(k+3)}+{8\over (k+1)(k+2)^2}H_{k+2},
    \end{align*}
    where $H_{k+2}$ is harmonic number defined by \eqref{eq:HarmonicNumbers}.
\end{corollary}

\begin{remark}
    We note, that by direct computations it is easy to ensure that
    \[
    \EE V_1^k={2\over (k+1)(k+2)}={1\over 2^{k}}\EE|Y_1-Y_2|^k,
    \]
    where $Y_1,Y_2$ are independent random variables uniformly distributed in $[0,1]$. This relation has a simple geometric meaning. Namely let $\xi$ be a height of the random triangle $T_1$, then it is easy to ensure by direct computations that $\sqrt{2}\xi$ has the same distribution as $|Y_1-Y_2|$.
\end{remark}

Let $a_j={2\over j(j+1)}$, $j\in\mathbb{N}$. In what follows it will be convenient to define 
\begin{equation}\label{eq:ConstantsC}
c(k,\ell):=e_k(a_1,\ldots,a_{\ell})=\sum_{1\leq j_1<\ldots<j_k\leq \ell}\prod_{i=1}^k{2\over j_i(j_i+1)},\qquad k,\ell\ge 0,
\end{equation}
where $e_k$, $k\ge 0$ is an elementary symmetric polynomial of degree $k$. We also note that $c(0,\ell)=e_0(a_1,\ldots,a_{\ell})=1$ for any $\ell\ge 0$ and $c(k,\ell)=0$ for $k>\ell$.

\begin{theorem}\label{tm:FunctionG}
    For any $n\in\NN$ and $t\in\RR$ we have
     \begin{align*}
    G_{n}(t)&=1+\sum_{\ell=1}^{n}(-1)^{\ell}{n\choose \ell}\prod_{i=1}^{\ell}\Big(1-{2t\over i(i+1)}\Big)=\sum_{k=1}^{n}t^k\sum_{\ell=k}^{n}(-1)^{\ell+k}{n\choose\ell}c(k,\ell).
    \end{align*}
\end{theorem}

As a direct corollary of above theorem we obtain a formula for the probabilities $p_{k}^{(n)}$, which is an alternative to \eqref{eq:Probabilities}.

\begin{corollary}\label{cor:Probabilities}
    For any $n\in\mathbb{N}$ and $0\leq k\leq n$ we have
    \[
    p_{k}^{(n)}=\sum_{\ell=k}^{n}(-1)^{\ell+k}{n\choose\ell}c(k,\ell).
    \]
\end{corollary}

The next formula for the moments of the volume of the convex chain is now obtained as a combination of Corollary \ref{cor:Probabilities} with Proposition \ref{lm:RecurrenceAP}.

\begin{theorem}\label{tm:MomentFormulaExact}
    For any $n\ge 1$ and $k\ge 0$ we have
    \begin{align}\label{eq:MomentsFormulaExact}
    \EE V_n^k&=\sum_{i=1}^n\sum_{\ell=i}^{n+k}(-1)^{\ell+i}{{n+k-i\choose k}{n+k\choose \ell}\over {n+k\choose k}}c(i,\ell).
    \end{align}
\end{theorem}

Analysing the above expression we obtain the following corollary, providing more precise formulas for $\E V_n$ and $\E V_n^2$ is terms of harmonic numbers.

\begin{corollary}\label{cor:ExpectedVolume2}
For any $n\ge 0$ we have
    \begin{align*}
    \EE V_n&=
    1-{1\over 3(n+1)}-{2H_{n+1}\over 3(n+1)},\\
    \EE V_n^2&=1-\frac{18n^2 + 106n + 144}{27 (n+1)(n+2)(n+3)}
-{(36n+32)H_{n+2}\over 27(n+1)(n+2)}+{4(H^2_{n+2}+H^{(2)}_{n+2})\over 9(n+1)(n+2)},
    \end{align*}
    where $H_n$ and $H_n^{(2)}$ are defined by \eqref{eq:HarmonicNumbers}.
\end{corollary}

\begin{remark}
Using the same analysis as in the proof of Corollary \ref{cor:ExpectedVolume2} one could also derive the formulas for $V_n^3$, $V_n^4$, etc. in terms of rational combination of harmonic and generalized harmonic numbers. The expressions become very long though.
\end{remark}

\begin{remark}
Using the fact that $H_{n+k}=\log n+O(1)$ and $H_{n+k}^{(2)}=O(1)$ as $n\to\infty$ for any fixed $k\ge 0$ we get that $\E V_n=1-{2\log n\over 3n}+O(n^{-1})$ and $\E V_n^2=1-{4\log n\over 3 n}+O(n^{-1})$ as $n\to\infty$. We conjecture, that it holds in general that $\E V_n^k=1-{2k\log n\over 3n}+O(n^{-1})$ for any fixed $k\ge 1$ as $n\to\infty$.
\end{remark}

\begin{remark}
Let us consider a random variable 
\[
D_n:=1-V_n=2\vol(T\setminus T_n),
\]
describing the normalized volume of $T$ not covered by $T_n$ (the so-called missed volume). It is easy to see that the moments of $D_n$ and the moments of $V_n$ are related via the equality $\E D_n^k=\sum_{\ell=0}^k{k\choose \ell}(-1)^{\ell}\E V_n^\ell$. In particular from Corollary \ref{cor:ExpectedVolume2} we deduce
\begin{align*}
\E D_n&=1-\E V_n={2H_{n+1}\over 3(n+1)}+{1\over 3(n+1)}={2\log n\over 3n}+O(n^{-1}),\\
\E D_n^2&={40H_{n+2}\over 27(n+1)(n+2)}-\frac{16n + 36}{27 (n+1)(n+2)(n+3)}
+{4(H^2_{n+2}+H^{(2)}_{n+2})\over 9(n+1)(n+2)}={4(\log n)^2\over 9 n^2}+O(n^{-2}).
\end{align*}
\end{remark}

\subsection{Proofs}

\begin{proof}[Proof of Theorem \ref{thm:MomentRecursion}]
We start by noting that since for any $x,y\in\RR$ with $|x|+|y|<1$ the series
\[
\widetilde Q(x,y)=\sum_{k=0}^\infty \sum_{n=0}^\infty (n+k)(n+k+1)q_{n,k} x^n y^k
\]
converges absolutely (see Lemma \ref{lm:QtildeLemma}) we get that
\[
(n+k)(n+k+1)q_{n,k}=\dfrac{1}{n!}\dfrac{1}{k!}\dfrac{\partial^{n+k}}{\partial x^{n}\partial y^{k}}\widetilde Q(x,y)\Big|_{x=y=0}.
\]
Further, by \eqref{eq::widePtssumGl} and using the product rule we obtain
\begin{align*}
(n+k)(n+k+1)q_{n,k}&=
\dfrac{2}{n!k!}\dfrac{\partial^{n+k}}{\partial x^{n}\,\partial y^{k}}\Big[\dfrac{x }{(x+y-1)^2}Q(x,y)\Big]\Big|_{x=y=0}\\
&=
\dfrac{2}{n!k!}\dfrac{\partial^{n}}{\partial x^{n}}\Big[\sum_{\ell=0}^k{k\choose \ell}\Big(\dfrac{\partial ^{\ell}}{\partial y^{\ell}}Q(x,y)\Big|_{y=0}\Big)x(x-1)^{-k+\ell-2}(-1)^{k-\ell}(k-\ell+1)!\Big]\Big|_{x=0}\\
&=
\dfrac{2}{n!k!}\sum_{\ell=0}^{k}{k\choose \ell}\sum_{n_1+n_2+n_3=n}\dfrac{n!}{n_1!n_2!n_3!}\Big(\dfrac{\partial ^{\ell+n_1}}{\partial y^{\ell}\partial x^{n_1}}Q(x,y)\Big|_{x=y=0}\Big)\\
&\hspace{6cm}\times (k-\ell+1+n_2)!\Big({\partial^{n_3} \over \partial x^{n_3}}x\Big|_{x=0}\Big).
\end{align*}
We note that ${\partial^{n_3} \over \partial x^{n_3}}x\big|_{x=0}={\bf 1}\{n_3=1\}$ and, hence, by setting $n_1=m$, $n_2=n-1-m$ and using \eqref{eq::Ankk!n!} we get
\begin{align*}
&(n+k)(n+k+1)q_{n,k}=2\sum_{\ell=0}^{k}\sum_{m=0}^{n-1}\dfrac{(k+n-\ell-m)!}{(k-\ell)!(n-1-m)!}q_{m,\ell}.
\end{align*}
Finally noting that $
q_{n,k}={{n+k}\choose {k} }\E V_n^k
$ we obtain
\begin{align*}
&(n+k)(n+k+1){(n+k)!\over n!k!}\E V_n^k=2\sum_{\ell=0}^{k}\sum_{m=0}^{n-1}\dfrac{(k+n-\ell-m)!}{(k-\ell)!(n-1-m)!}{(m+\ell)!\over m!\ell!}\EE V_m^{\ell},
\end{align*}
and rearranging coefficients we conclude
\begin{align*}
\E V_n^k&=\dfrac{2n}{(n+k)(n+k+1)}\sum_{\ell=0}^{k}\sum_{m=0}^{n-1}{(n-1)!\over m!(n-1-m)!}{k!\over \ell!(k-\ell)!}{(k+n-\ell-m)!(m+\ell)!\over (k+n)!}\EE V_m^{\ell}\\
&=\dfrac{2n}{(n+k)(n+k+1)}\sum_{\ell=0}^{k}\sum_{m=0}^{n-1}\dfrac{{{n-1}\choose{m}}{{k}\choose{\ell}}}{{{n+k}\choose{m+\ell}}}\E V_{m}^{\ell}.
\end{align*}

\end{proof}

\begin{proof}[Proof of Corollary \ref{cor:ExpectedVolume1}]

    The formula for $\EE V_1^k$ follows directly by \eqref{eq::EVsum}, namely
    \[
    \EE V_1^k={2\over (k+1)(k+2)}\sum_{\ell=0}^1\EE V_0^{\ell}={2\over (k+1)(k+2)},
    \]
    since $\EE V_0^{\ell}=0$ for $\ell\ge 1$ and $\EE V_0^{0}=1$. Similarly for $\EE V_2^k$ the application of \eqref{eq::EVsum} leads to
    \begin{align*}
        \EE V_2^k&=\dfrac{4}{(k+2)(k+3)}\Big(1+\sum_{\ell=0}^k\dfrac{{{k}\choose{\ell}}}{{{k+2}\choose{\ell+1}}}\EE V_1^{\ell}\Big)\\
        &=\dfrac{8}{(k+1)(k+2)^2(k+3)}\sum_{\ell=0}^k{k+3-(\ell+2)\over \ell+2}+{4\over (k+2)^2}\\
        &={4(k-3)\over(k+1)(k+2)(k+3)}+{8\over (k+1)(k+2)^2}\sum_{i=1}^{k+2}{1\over i}.
    \end{align*}
\end{proof}

\begin{proof}[Proof of Theorem \ref{tm:FunctionG}]
    By Theorem \ref{tm:GenrationFunction} for any $t\in\RR$ and $z\in\RR$ with $|z|<(|t|+|1-t|)^{-1}$ we have
    \[
    L(t;z)=\sum_{\ell=0}^{\infty}z^{n}G_{n}(t)=F\Big({1\over 2}+{1\over 2}\sqrt{1+8t},{1\over 2}-{1\over 2}\sqrt{1+8t};2;{z\over z-1}\Big),
    \]
    and the series converges absolutely. Given functions $f,g$ recall the Fa\'a di Bruno's formula 
    \begin{equation}\label{eq:FaaDiBruno}
    {{\rm d}^{n}\over {\rm d} z^{n}} f(g(z))=\sum_{\ell=1}^{n}{{\rm d}^\ell\over {\rm d} z^{\ell}} f(g(z))\cdot B_{n,\ell}\Big({{\rm d}\over {\rm d} z} g(z), {{\rm d}^2\over {\rm d} z^{2}}g(z),\ldots, {{\rm d}^{n-\ell+1}\over {\rm d} z^{n-\ell+1}}g(z)\Big),
    \end{equation}
    where 
    \[
    B_{n,\ell}(x_1,x_2,\ldots,x_{n-\ell+1})=n!\sum_{\substack{j_1,\ldots, j_{n-\ell+1}\ge 0,\\
    j_1+\ldots+j_{n-\ell+1}=\ell,\\
    j_1+2j_2+\ldots +(n-\ell+1)j_{n-\ell+1}=n}}\prod_{i=1}^{n-\ell+1}{1\over j_i!}\Big({x_i\over i!}\Big)^{j_i}
    \]
    is an incomplete Bell polynomial. We will apply \eqref{eq:FaaDiBruno} for $f(z)=F\big({1\over 2}+{1\over 2}\sqrt{1+8t},{1\over 2}-{1\over 2}\sqrt{1+8t};2;z\big)$ and $g(z)={z\over z-1}=1+(z-1)^{-1}$. First note that for any $n\ge 1$ we have
    \begin{align*}
    {{\rm d}^n\over {\rm d} z^{n}}g(z)\Big|_{z=0}=-n!,
    \end{align*}
    and by \eqref{eq:HypergeometricFuncDeriv} together with $F(a,b;c;0)=1$ for any $a,b,c\in\CC$, $c\neq 0,-1,\ldots$ we get
    \begin{align*}
        {{\rm d}^n\over {\rm d} z^{n}}f(z)\Big|_{z=0}={({1\over 2}+{1\over 2}\sqrt{1+8t})_n({1\over 2}-{1\over 2}\sqrt{1+8t})_n\over (n+1)!}={1\over (n+1)!}\prod_{i=0}^{n-1}(i(i+1)-2t).
    \end{align*}
     Combining these observations together we obtain
    \begin{align*}
    G_{n}(t)={1\over n!}{{\rm d}^{n}\over {\rm d} z^{n}}L(t;z)\Big|_{z=0}&={1\over n!}\sum_{\ell=1}^{n}{1\over (\ell+1)!}\prod_{i=0}^{\ell-1}(i(i+1)-2t)B_{n,\ell}(-1!,-2!,\ldots, -(n-\ell+1)!)\\
    &=\sum_{\ell=1}^{n}{(-1)^\ell\over \ell!(\ell+1)!}{n-1\choose \ell-1}\prod_{i=0}^{\ell-1}(i(i+1)-2t)\\
    &=t-\sum_{\ell=2}^{n}(-1)^\ell{n-1\choose \ell-1}{2t\over \ell(\ell+1)}\prod_{i=1}^{\ell-1}\Big(1-{2t\over i(i+1)}\Big),
    \end{align*}
    where in the second step we used that $B_{n,\ell}(-x_1,\ldots, -x_{n-\ell+1})=(-1)^{\ell}B_{n,\ell}(x_1,\ldots, x_{n-\ell+1})$ and 
    \[
    B_{n,\ell}(1!,2!,\ldots, (n-\ell+1)!)={n-1\choose \ell-1}{n!\over \ell!}.
    \]
    By performing a few more transformations we obtain
    \begin{align*}
    G_{n}(t)&=t+\sum_{\ell=2}^{n}(-1)^\ell{n-1\choose \ell-1}\prod_{i=1}^{\ell}\Big(1-{2t\over i(i+1)}\Big)+\sum_{\ell=2}^{n}(-1)^{\ell-1}{n-1\choose \ell-1}\prod_{i=1}^{\ell-1}\Big(1-{2t\over i(i+1)}\Big)\\
    &=t+\sum_{\ell=2}^{n}(-1)^\ell{n-1\choose \ell-1}\prod_{i=1}^{\ell}\Big(1-{2t\over i(i+1)}\Big)+\sum_{\ell=1}^{n-1}(-1)^{\ell}{n-1\choose \ell}\prod_{i=1}^{\ell}\Big(1-{2t\over i(i+1)}\Big)\\
    &=\sum_{\ell=1}^{n}(-1)^\ell{n-1\choose \ell-1}\prod_{i=1}^{\ell}\Big(1-{2t\over i(i+1)}\Big)+\sum_{\ell=1}^{n-1}(-1)^{\ell}{n-1\choose \ell}\prod_{i=1}^{\ell}\Big(1-{2t\over i(i+1)}\Big)+1\\
    &=\sum_{\ell=0}^{n}(-1)^\ell{n\choose \ell}\prod_{i=1}^{\ell}\Big(1-{2t\over i(i+1)}\Big),
    \end{align*}
    where we used ${n-1\choose \ell-1}+{n-1\choose \ell}={n\choose \ell}$ and the fact that the empty product equals $1$. This finishes the proof of the first equality.

    The second equality follows by opening the brackets in the first equality, more precisely
    \begin{align*}
        G_{n}(t)&=\sum_{\ell=0}^{n}(-1)^\ell{n\choose \ell}\prod_{i=1}^{\ell}\Big(1-{2t\over i(i+1)}\Big)\\
        &=\sum_{\ell=0}^{n}(-1)^\ell{n\choose \ell}\sum_{k=0}^\ell(-t)^k\prod_{1\leq j_1<\ldots<j_k\leq \ell}{2\over j_i(j_i+1)}\\
        &=\sum_{k=0}^{n}t^k\sum_{\ell=k}^{n}(-1)^{\ell+k}{n\choose \ell}c(k,\ell).
    \end{align*}
\end{proof}

\begin{proof}[Proof of Theorem \ref{tm:MomentFormulaExact}]
    Combining Corollary \ref{cor:Probabilities} with Proposition \ref{lm:RecurrenceAP} we immediately get 
    \begin{align*}
        \EE V_n^k&=\sum_{i=0}^n {{{n+k-i}\choose k}\over {n+k\choose k}} p_i^{(n+k)}=\sum_{i=0}^n\sum_{\ell=i}^{n+k}(-1)^{\ell+i}{{n+k-i\choose k}{n+k\choose \ell}\over {n+k\choose k}}c(i,\ell)=\sum_{i=1}^{n}\sum_{\ell=i}^{n+k}(-1)^{\ell+i}{{n+k-i\choose k}{n+k\choose \ell}\over {n+k\choose k}}c(i,\ell),
    \end{align*}
    since $c(0,\ell)=1$ for any $\ell\ge 0$ and $\sum_{\ell=0}^{n+k}(-1)^{\ell}{n+k\choose \ell}=0$.
\end{proof}

\begin{proof}[Proof of Corollary \ref{cor:ExpectedVolume2}]
Recall, that $c(i,\ell)$ is an elementary symmetric polynomial of degree $i$ in $\ell$ variables $a_1,\ldots,a_{\ell}$. Denote by 
\[
C_{\ell}(t):=\sum_{i=0}^{\ell}t^i c(i,\ell)=\prod_{j=1}^{\ell}(1+a_jt),
\]
the generating function of the sequence $(c(i,\ell))_{i=0}^{\ell}$. We note that for any $m\ge 1$ it holds that
\[
C^{(m)}_{\ell}(t)=(-1)^m\sum_{i=1}^{\ell}(-i)_mt^{i-m}c(i,\ell),
\]
where we recall that $(x)_m=x(x+1)\ldots (x+m-1)$ is rising factorial. In particular it holds
\begin{align}
C_{\ell}(-1)-1&=\sum_{i=1}^{\ell}(-1)^ic(i,\ell)=\prod_{j=1}^{\ell}{(j-1)(j+2)\over j(j+1)}-1=-1,\label{eq:Am}\\
C_{\ell}'(-1)&=\sum_{i=1}^{\ell}(-1)^i(-i)c(i,\ell)=\prod_{j=2}^{\ell}{(j-1)(j+2)\over j(j+1)}={\ell+2\over 3\ell},\label{eq:Aprime}\\
C_{\ell}^{(m)}(-1)&=\sum_{i=1}^{\ell}(-1)^i(-i)_mc(i,\ell).\notag
\end{align}

We start by analysing the formula \eqref{eq:MomentsFormulaExact} in order to get
\begin{align*}
    \EE V_n^k&={n!\over (n+k)!}\sum_{i=1}^n(n+1-i)_k\sum_{\ell=i}^{n+k}(-1)^{\ell+i}{n+k\choose \ell}c(i,\ell)\\
    &={n!\over (n+k)!}\sum_{\ell=1}^{n+k}(-1)^{\ell}{n+k\choose \ell}\sum_{i=1}^{\ell}(n+1-i)_k(-1)^{i}c(i,\ell).
\end{align*}
Note, that $(n+1-i)_k=0$ for $i\ge n+1$. Further using the formula
\[
(n+1-i)_k=\sum_{m=0}^k{k\choose m}(n+1)_{k-m}(-i)_m
\]
we obtain
\begin{align*}
    \EE V_n^k&={n!\over (n+k)!}\sum_{\ell=1}^{n+k}(-1)^{\ell}{n+k\choose \ell}\sum_{i=1}^{\ell}\sum_{m=0}^k{k\choose m}(n+1)_{k-m}(-1)^{i}(-i)_mc(i,\ell)\\
    &={n!\over (n+k)!}\sum_{\ell=1}^{n+k}(-1)^{\ell}{n+k\choose \ell}\sum_{m=0}^k{k\choose m}(n+1)_{k-m}\sum_{i=1}^{\ell}(-1)^{i}(-i)_mc(i,\ell)\\
    &={n!\over (n+k)!}\sum_{\ell=1}^{n+k}(-1)^{\ell}{n+k\choose \ell}\Big((n+1)_k(C_{\ell}(-1)-1)+\sum_{m=1}^k{k\choose m}(n+1)_{k-m}C_{\ell}^{(m)}(-1)\Big)\\
    &=1+{n!\over (n+k)!}\sum_{m=1}^k{k\choose m}(n+1)_{k-m}\sum_{\ell=1}^{n+k}(-1)^{\ell}{n+k\choose \ell}C_{\ell}^{(m)}(-1),
\end{align*}
where in the last step we used \eqref{eq:Am} and the identity $\sum_{\ell=1}^{n+k}(-1)^{\ell}{n+k\choose \ell}=-1$.

Taking $k=1$ and using \eqref{eq:Aprime} we arrive at
\begin{align*}
\EE V_n&=1+{1\over (n+1)}\sum_{\ell=1}^{n+1}(-1)^{\ell}{n+1\choose \ell}\Big({1\over 3}+{2\over 3\ell}\Big).
\end{align*}
Further using identities $\sum_{\ell=1}^{n+1}(-1)^{\ell}{n+1\choose \ell}=-1$ and $\sum_{\ell=1}^{n+1}(-1)^{\ell-1}{1\over \ell}{n+1\choose \ell}=H_{n+1}$ we get
\[
\EE V_n=1-{1\over 3(n+1)}-{2\over 3(n+1)}H_{n+1}.
\]

For $k=2$ we need to analyse $C_{\ell}''(-1)$. We get
\begin{align*}
C_{\ell}^{''}(-1)&=2\sum_{1\leq j_1<j_2\leq \ell}\prod_{j=1,j\neq j_1,j_2}^{\ell}{(j-1)(j+2)\over j(j+1)}{2\over j_1(j_1+1)}{2\over j_2(j_2+1)}\\
&=2\sum_{2\leq j_2\leq \ell}\prod_{j=2}^{\ell}{(j-1)(j+2)\over j(j+1)}{2\over (j_2-1)(j_2+2)}\\
&={2(\ell+2)\over 3\ell}\Big({2\over 3}\sum_{i=1}^{\ell-1}{1\over i}-{2\over 3}\sum_{i=4}^{\ell+2}{1\over i}\Big)\\
&={22\over 27}-{4\over 27\ell}+{4\over 9(\ell+1)}-{8\over 9\ell^2}.
\end{align*}
Now using the identities \cite[Equation (3.3) and (3.4)]{Batir}
\begin{align*}
 \sum_{\ell=1}^{n+2}(-1)^\ell{n+2\choose \ell}&=-1,\qquad &\sum_{\ell=1}^{n+2}(-1)^{\ell-1}{1\over \ell+1}{n+2\choose \ell}&={n+2\over n+3},\\
 \sum_{\ell=1}^{n+2}(-1)^{\ell-1}{1\over \ell}{n+2\choose \ell}&=H_{n+2},\qquad &\sum_{\ell=1}^{n+2}(-1)^{\ell-1}{1\over \ell^2}{n+2\choose\ell}&={H^2_{n+2}+H^{(2)}_{n+2}\over 2}
 \end{align*}
we obtain
\begin{align*}
    \EE V_n^2&=1+{1\over (n+1)(n+2)}\Big(2(n+1)\sum_{\ell=1}^{n+2}(-1)^\ell{n+2\choose \ell}\Big({1\over 3}+{2\over 3\ell}\Big)\\
    &\qquad\qquad+\sum_{\ell=1}^{n+2}(-1)^\ell{n+2\choose \ell}\Big({22\over 27}-{4\over 27\ell}+{4\over 9(\ell+1)}-{8\over 9\ell^2}\Big)\Big)\\
&=1-\frac{18n^2 + 106n + 144}{27 (n+1)(n+2)(n+3)}
-{(36n+32)H_{n+2}\over 27(n+1)(n+2)}+{4(H^2_{n+2}+H^{(2)}_{n+2})\over 9(n+1)(n+2)}.
\end{align*}
\end{proof}

\section*{Acknowledgements}

AG was supported by the DFG priority program SPP 2265 \textit{Random Geometric Systems} and Germany's Excellence Strategy  EXC 2044 -- 390685587, \textit{Mathematics M\"unster: Dynamics - Geometry - Structure}.
    
\bibliography{Bibliography}

\end{document}